\documentclass{article}
\usepackage{amsmath}
\usepackage{amsthm}
\usepackage{mathrsfs}
\usepackage{amsfonts}
\usepackage{amsopn}
\DeclareMathOperator{\rank}{rank}
\begin{document}
\title{Cleavability and scattered sets of non-trivial fibers}
\author{Shari Levine}
\date{June 7, 2011}
\maketitle
\hrulefill

\newtheorem{thm}{Theorem}[section]
\newtheorem{lemma}[thm]{Lemma}
\newtheorem{definition}[thm]{Definition}
\newtheorem{obs}[thm]{Observation}
\newtheorem{question}{Question}
\newtheorem{claim}[thm]{Claim}
\newtheorem{prop}[thm]{Proposition}

\section{Introduction}

A space $X$ is said to be \textit{cleavable} over a space $Y$ \textit{along} $A \subseteq X$ if there exists a continuous $f: X \rightarrow Y$ such that $f(A) \cap f(X \setminus A) = \emptyset$. A space $X$ is \textit{cleavable over} $Y$ if it is cleavable over $Y$ along all $A \subseteq X$. The subject was introduced by A.~V.~Arhangel'ski\u\i\ and D.~B.~Shakh\-matov in \cite{splitting}, though it was originally termed \textit{splitting}, and it was in \cite{Arhangelskii} that A.~V.~Arhangel'ski\u\i\ posed the main questions related to the study of cleavability:

\begin{question} When does cleavability of a space $X$ over a Hausdorff space $Y$ imply the existence of a homeomorphism from $X$ to a subspace of $Y$? \end{question} 

\begin{question} \label{q} Let $X$ be an infinite compactum cleavable over a linearly ordered topological space (LOTS). Is $X$ a LOTS? \end{question}

Results related to these questions can be found in, but are not limited to, the following papers: \cite{Arhangelskii}, \cite{Arhangelskii2}, \cite{buzyakova}.

It is customary in this field that if $f$ is a continuous function, then we represent the set of points on which $f$ is not injective as $M_f$.

In this paper, we show that if $X$ is a compact space cleavable over a separable LOTS $Y$ such that for some continuous $f:X \rightarrow Y$, $M_f$ is scattered, then $X$ is a LOTS. We do so by first considering the case when $X$ is totally disconnected (Section $2$), and then use that result to prove it for any compact $X$ (Section $3$). 

\section{Totally Disconnected $X$}

In this section we show that if $X$ is a totally disconnected compact space cleavable over a separable LOTS $Y$ such that for some continuous $f:X \rightarrow Y$, $M_f$ is scattered, then $X$ is a LOTS. We do so by showing that there exists a LOTS $\hat{Y}$ and an injective continuous function $\hat{f}$ mapping $X$ into $\hat{Y}$. As $\hat{f}$ is a closed map, $\hat{f}$ must be a homeomorphism, making $X$ a closed subspace of a LOTS, and therefore a LOTS itself. The main results of this section is given by Theorems~\ref{scattered} and \ref{result}, with the rest of the section containing tools needed for the proofs of the aforementioned theorems. The most important of these tools are Lemmas~\ref{partition} and \ref{less}, and we explain their importance before the statements of the lemmas. 

Proving that $X$ is a LOTS when it is totally disconnected is crucial to the proof for when $X$ is not assumed to be totally disconnected. Before we prove either, however, it is important to provide a few definitions. 

\begin{definition} Let $X$ be a topological space, and let $A$ be a subset of $X$. The \textbf{derived set} of $A$, written as $A'$, is the set of all limit points of $A$. \end{definition}

\begin{definition} For ordinal numbers $\alpha$, the $\alpha$-th \textbf{Cantor-Bendixson derivative} of a topological space $X$ is defined by transfinite induction as follows:
\begin{itemize}
\item $X^0 = X$
\item $X^{\alpha+1} = (X^{\alpha})'$
\item $\displaystyle X^{\lambda} = \bigcap_{\alpha < \lambda} X^{\alpha}$ for limit ordinals $\lambda$.\end{itemize}
The smallest ordinal $\alpha$ such that $X^{\alpha+1} = X^{\alpha}$ is called the \textbf{Cantor-Bendixson rank of $X$}, written as $\rank(X)$. The least ordinal $\beta$ such that $x \notin X^{\beta}$ is called the \textbf{rank of $x$}, written as $\rank(x)$ \end{definition}

To clarify, if we say the rank of $x$ is $\beta+1$, we mean that $X^{\beta}$ is the last derived set of $X$ of which $x$ is an element. 

\begin{definition} Let $X$ be a compact LOTS, and let $A \subset X$ be closed. We say a non-empty open interval $(a,b) \subset X \setminus A$ is \textbf{maximal} if either $a$ and $b$ are both elements of $A$, or one is an element of $A$, and the other is an endpoint of $X$. \end{definition}

The following proposition and theorem are from \cite{buzyakova} and \cite{contributions} respectively.

\begin{prop} \label{basics} If $X$ is a compact $T_2$ space cleavable over a separable LOTS $Y$, then $X$ is separable and first-countable. \end{prop} 

\begin{thm} \label{contributions} If $X$ is a countable compact metric space then $X$ is homeomorphic to a countable ordinal. \end{thm}

\begin{lemma} If $X$ is a countable, compact $T_2$ space cleavable over a separable LOTS $Y$, then every closed $A \subseteq X$ is homeomorphic to a countable ordinal. \end{lemma}

\begin{proof} This follows directly from Proposition~\ref{basics} and Theorem~\ref{contributions}. \end{proof}

We now have some information about any compact totally disconnected $X$ cleavable over a separable LOTS $Y$, and this will help us prove the important lemmas of this section, Lemmas~\ref{partition} and \ref{less}. However, it should be explained why these lemmas are so significant.

We want to answer Question~\ref{q} in the affirmative. If we had an injective and continuous $f$ from $X$ to $Y$, we would have an immediate answer, as $f$ would be an embedding, and since $f(X)$ would be a closed subspace of a LOTS, $X$ would be a LOTS as well. In a very informal sense, there are two reasons why we may not be able to find an injective map from $X$ to $Y$. Either the topology on $X$ is too complicated for the elements of $X$ to be linearly ordered, or there isn't ``enough room'' in $Y$ to continuously and injectively map all of the points of $X$. What Lemmas~\ref{partition} and \ref{less} ensure is that for any single $y \in f(M_f)$, we may find a LOTS $\hat{Y}$ and a continuous $f: X \rightarrow \hat{Y}$ with enough room to accommodate the points of $f^{-1}(y)$. Since we will be assuming $M_f$ is scattered for some $f$, we will eventually be able to systematically repeat the method contained in Lemma~\ref{less} to find a $\hat{Y}$ that accommodates all points of $M_f$.

What Lemma~\ref{partition} actually does is strategically partitions $X$ so that, when we do find a $\hat{Y}$ with ``more room'' than $Y$ (Lemma~\ref{less} finds this $\hat{Y}$), our function from $X$ to $\hat{Y}$ is continuous. 

\begin{lemma} \label{partition} Let $X$ be a totally disconnected, first-countable, compact $T_2$ space, and $A$ a countable, closed subset of $X$ such that $A$ is homeomorphic to some countable ordinal $\lambda$. Let $h: A \rightarrow \lambda$ be a homeomorphism, and let $A = \left\{x_{\beta}: \beta < \lambda \right\}$, where $h(x_{\beta}) = \beta$. For each $\alpha \in \lambda \setminus \lambda^1$, we may then find clopen sets $U_{\alpha} \ni x_{\alpha}$ such that the following are satisfied: \begin{enumerate}

\item For $\alpha_1 \neq \alpha_2$, where $\alpha_1,\alpha_2 \in \lambda \setminus \lambda^1$, $U_{\alpha_1} \cap U_{\alpha_2} = \emptyset$.

\item Let $C \subseteq \lambda$ be clopen. Then $\bigcup_{\alpha \in (\lambda \setminus \lambda^1) \cap C} U_{\alpha} \cup \left\{x_{\beta}: \beta \in \lambda^1 \cap C\right\}$ is clopen.

\item $\bigcup_{\alpha \in \lambda \setminus {\lambda}^1} U_{\alpha} \cup \left\{x_{\beta}: \beta \in {\lambda}^1\right\} = X$. \end{enumerate} \end{lemma}

\begin{proof} We will prove this by transfinite induction on the Cantor-Bendixson rank of the last non-empty derived set of $A$.\\

\textbf{Base Case}: Let $A^0 = A$ be the last non-empty derived set of $A$. That is, all elements of $A$ are isolated in $A$. Since $A$ is assumed to be closed, we know there are finitely many elements in $A$, and since $X$ is totally disconnected and $T_2$, we know we may partition $X$ into finitely many clopen sets, each containing only one element of $A$. These clopen sets clearly satisfy all required properties.\\

\textbf{Successor Case}: It will be easiest to prove the theorem true for the case when the last non-empty derived set of $A$ has rank $1$, and then use this result to prove the general successor case. 

Let the last non-empty derived set of $A$ be $A^1$, and without loss of generality, assume there is only one element in $A^1$, namely $x_{\omega}$. Enumerate the elements of $A \setminus A^1$ as $x_n$, where $n \in \omega$. Since $X$ is first-countable and zero-dimensional, we may find a countable local base for $x_{\omega}$ consisting of nested clopen sets, $\left\{D_n: n \in \omega\right\}$, such that $D_0 = X$ and such that $\bigcap_{n \in \omega} D_n = \left\{x_{\omega}\right\}$. We may also require that $A \cap (X \setminus D_j) = \left\{x_k: k < j\right\}$; that is, $x_0 \in X \setminus D_1$, $x_0,x_1 \in X \setminus D_2$, and so on. Notice that for each $j$, $C_j = (X \setminus D_{j+1}) \setminus (X \setminus D_j)$ is a clopen set containing $x_j$ and no other elements of $A \setminus A^1$. Furthermore, $\bigcup_{n \in \omega} C_n \cup \left\{x_{\omega}\right\} = X$. These sets, $C_n$, obviously satisfy the theorem's requirements.

For the general successor case, let the last non-empty derived set of $A$ be $A^{\alpha+1}$, and without loss of generality, assume that $A^{\alpha+1}$ has only one element, $x_{\alpha+1}$. If we let $B = A^{\alpha}$, then $B^1$ must be the last non-empty derived set; we know, therefore, from the first part of the successor-case proof that we may partition $X$ in such a way that $X = \left\{x_{\alpha+1}\right\} \cup \bigcup_{n \in \omega} D_{x_{\alpha,n}}$, where each $D_{x_{\alpha,n}}$ is clopen, and contains exactly one element of $B\setminus B^1$, namely $x_{\alpha,n}$. When we consider this clopen set with respect to $A$ again, we notice that the rank of $A \cap D_{x_{\alpha,n}}$ is less than $\alpha+1$, and contains only one element of $A^{\alpha}$: $x_{\alpha,n}$. By the inductive hypothesis, we know we may partition each $D_{x_{\alpha,n}}$ in such a way that fits the theorem's requirements. But does the collective partitioning, the one in which we consider all clopen sets created from partitioning each $D_{x_{\alpha},n}$, satisfy the theorem's three requirements? It is obvious that this partition satisfies requirements $1$ and $3$. We must now check property $2$ is satisfied.

To clarify notation, consider $A \cap D_{x_{\alpha,0}}$. This set is homeomorphic to an ordinal $\delta+1$; in fact, all $A \cap D_{x_{\alpha,n}}$ are homeomorphic to $\delta+1$. When we use the inductive hypothesis to partition $D_{x_{\alpha,0}}$, based on $A \cap D_{x_{\alpha,0}}$, we will write the clopen subsets as $U_{\alpha}$, where $\alpha \in (\delta+1) \setminus (\delta+1)^1$. Now consider $D_{x_{\alpha,1}}$. When we use the inductive hypothesis to partition $D_{x_{\alpha,1}}$, based on $A \cap D_{x_{\alpha,1}}$, we will again write the clopen subsets as $U_{\alpha}$, but our $\alpha$ will now range over $[\delta+1,\delta \cdot 2 +1) \setminus ([\delta+1,\delta \cdot 2 +1))^1$. To generalize, when we partition $D_{x_{\alpha,n}}$ using the inductive hypothesis, we write our clopen sets as $U_{\alpha}$, where $\alpha \in [\delta \cdot n +1,\delta \cdot (n+1) +1) \setminus ([\delta \cdot n +1,\delta \cdot (n+1) +1))^1$. Lastly, let $\bigcup_{\alpha \in (\lambda \setminus \lambda^1) \cap C} U_{\alpha} \cup \left\{x_{\beta}: \beta \in \lambda^1 \cap C\right\}$, the set described in requirement $2$ of the theorem, be labeled $E_C$.

Let $\beta+1$ be the ordinal to which $A$ is homeomorphic (as $A$ is closed). Firstly notice that a single clopen interval, $[a+1,b]$, is the complement of $[0,a] \cup [b+1,\beta]$, both of which are clopen as well. (Note that $a$ may equal $0$, and if $b = \beta$, then $[b+1,\beta]$ will be empty.) Further, if $[c+1,\beta]$ is a clopen interval, then $[0,c]$ is a clopen interval. Lastly, all clopen subsets of $\beta+1$ are the union of at most finitely many clopen intervals. For these three reasons, to show requirement $2$ is satisfied for all clopen $C \subseteq \beta+1$, it is sufficient to show requirement $2$ of the theorem is satisfied whenever we take $C = [0,a]$, for some $a \in \beta+1$. 

Thus, let $C \subseteq \beta+1$ be equal to $[0,a]$. If $a \in \left\{0,\beta \right\}$, then $E_C$ is trivially clopen. Thus let $a \in [\delta \cdot m +1, \delta \cdot (m+1) + 1)$ for some $m \in \omega$. Then $E_C = \bigcup_{j<m} D_{x_{\alpha,j}} \cup \bigcup_{\mu \in ([\delta \cdot m +1,a]) \setminus ([\delta \cdot m + 1,a])^1} U_{\mu} \cup \left\{x_{\gamma}: \gamma \in ([\delta \cdot m +1,a])^1\right\}$. The set $\bigcup_{j<m} D_{x_{\alpha,j}}$ is clopen as it is the finite union of clopen sets, and $\bigcup_{\mu \in ([\delta \cdot m +1,a]) \setminus ([\delta \cdot m + 1,a])^1} U_{\mu} \cup \left\{x_{\gamma}: \gamma \in ([\delta \cdot m +1,a])^1\right\}$ is clopen by the inductive hypothesis. Therefore requirement $2$ is satisfied, and the successor case is proven.\\

\textbf{Limit Case}: Let the last non-empty derived set of $A$ have rank $\lambda$, let $\left\langle \alpha_n \right\rangle$ be a monotonically increasing sequence of successor ordinals converging to $\lambda$, and without loss of generality, let $A^{\lambda} = \left\{x_{\lambda}\right\}$. Let $x_{\alpha_n}$ be an element of $A$ with rank $\alpha_n$, and consider $B = \left\{x_{\alpha_n}: n \in \omega\right\} \cup \left\{x_{\lambda}\right\}$. This set is closed by compactness and first countability of $X$. Since $X$ is first-countable and zero-dimensional, we again know that there exists a countable local base for $x_{\lambda}$ consisting of nested clopen sets $\left\{D_n: n \in \omega\right\}$ such that $\bigcap_{n \in \omega} D_n = \left\{x_{\lambda}\right\}$, and such that $D_0 = X$. We may require this local base to be such that $B \cap X \setminus D_{x_{\alpha_{j+1}}} = x_{\alpha_0} \cup ... \cup x_{\alpha_j}$, for every $j+1 \in \omega$. 

Notice that for each $j+1$, $C_{j+1} = (X \setminus D_{j+1}) \setminus (X \setminus D_j)$ is a clopen set containing $x_{j+1}$ and no other elements of $B$. Furthermore, $\bigcup_{n \in \omega} C_n \cup \left\{x_{\omega}\right\} = X$. Since each $C_{j+1}$ is a clopen set, and $C_{j+1} \cap A$ has Cantor-Bendixson rank less than $\rank(A)$, we may use the inductive hypothesis to partition $C_{j+1}$ into sets that satisfy the requirements of the theorem. Now we must check that the collective partitioning satisfies for the limit case.

The first and third requirements are clearly met. The second requirement is proven in the exact same way as it was in the successor case.

This completes the proof. \end{proof}

\begin{lemma} \label{less} Let $X$ be a totally disconnected compact $T_2$ space cleavable over a separable LOTS $Y$, and $f: X \rightarrow Y$ a continuous function such that $M_f$ is countable. Then for every $x \in M_f$, there exists a separable LOTS $Y_1$ and a continuous function $f_1: X \rightarrow Y_1$ such that $x \notin M_{f_1}$, and such that $M_{f_1} \subset M_f$.\end{lemma}

\begin{proof} We will be relying on the notation used in Theorem~\ref{partition}. Let $x \in M_f$, let $y = f(x)$, and let $A = f^{-1}(y)$. Let $\lambda$ be the ordinal to which $A$ is homeomorphic, where $h: A \rightarrow \lambda$ is a homeomorphism, and enumerate the points of $A$ to be $x_{\alpha}$ such that $h(x_{\alpha}) = \alpha$. (It does not matter for the sake of the proof which $x_{\alpha}$ is equal to our original $x$.) Partition $X$ as described in Theorem~\ref{partition}, with $U_{\alpha} \ni x_{\alpha}$. Let $Y_1$ be $\lambda$ with all of the isolated ordinals $\alpha$ replaced with $f(U_{\alpha})$. Let the order on $Y_1$ preserve the order between the ordinals, and preserve the order already on $f(U_{\alpha}) \subset Y$. Let $g_{\alpha}: U_{\alpha} \rightarrow f(U_{\alpha}) \subset Y_1$ be identical to $f|_{U_{\alpha}}$, and let $h': A^1 \rightarrow \lambda^1 \subset Y_1$ be such that $h'(x_{\alpha}) = \alpha$. The functions $g_{\alpha}$ and $h'$ are clearly continuous. We first claim that there exists a continuous function $f_1: X \rightarrow Y_1$. 

Let $f_1$ be defined as: $$f_1(x) = \begin{cases}
g_{\alpha}(x) & x \in U_{\alpha}\\
h'(x) & x \in A^1\\ \end{cases}$$

First note that each $f_1(U_{\alpha})$ is clopen in $Y_1$. To show this function is continuous, let $V$ be an open set in $Y_1$. For every $y$ such that $y \in V \cap f_1(U_{\alpha})$ for some $\alpha$, we know there exists an open set in $f_1^{-1}(V)$ containing $x = f_1^{-1}(y)$, namely $f_1^{-1}(V \cap f(U_{\alpha}))$. (This is true by continuity of $f_1$ on $U_{\alpha}$.) If, however, for some $y \in V$, $y = h'(x_{\delta})$ for some $\delta \in \lambda^1$, then by construction we may find an ordinal $\gamma \in \lambda$ such that if $B = \left\{\alpha \in [\gamma,\delta): f(U_{\alpha}) \subset V \right\}$, and $C = \left\{\beta \in [\gamma, \delta)^1: f_1(x_{\beta}) \in V\right\}$, then $B \cup C \cup \left\{\delta\right\}$ is a clopen set of ordinals. This implies, by property two of Theorem~\ref{partition}, that $\bigcup_{\alpha \in B} U_{\alpha} \cup \left\{x_{\beta}: \beta \in C\right\} \cup \left\{x_{\delta}\right\}$ is a clopen set in $X$ containing $f_1^{-1}(y) = x_{\delta}$, contained in $f_1^{-1}(V)$. Therefore $f_1$ is continuous, $M_{f_1} \subset M_f$, and $x \notin M_{f_1}$.

Since $f_1$ is a continuous function from a separable space $X$ onto a LOTS $Y_1$, we know $Y_1$ must be separable. \end{proof}

We now have two powerful lemmas we may use to prove $X$ is a LOTS. The following two lemmas ensure that if $M_f$ is scattered in $X$ for some continuous $f$, then the points of $f(M_f)$ behave in $Y$ well enough for us to systematically implement the methods contained in Lemmas~\ref{partition} and \ref{less}.

\begin{lemma} Let $X$ be a totally disconnected, compact $T_2$ space cleavable over a totally disconnected separable LOTS $Y$. If there exists a continuous $f: X \rightarrow Y$ such that $M_f$ is scattered, then $f(M_f)$ is scattered in $Y$. \end{lemma}

\begin{proof} Let $f: X \rightarrow Y$ be the continuous function such that $M_f$ is scattered, and consider $f(M_f)$. Assume for a contradiction, and without loss of generality, that $f(M_f)$ is dense in itself. Let $y \in f(M_f)$, and consider $f^{-1}(y)$. This set contains an element $x$ such that for every other $x' \in f^{-1}(y)$, $\rank(x') \leq \rank(x)$; the rank we are referring to here and for the rest of this proof is its rank with respect to $M_f$. Since $f(X)$ is totally disconnected and zero-dimensional, we know there exists a clopen set $V_1$ containing $y$ such that $f^{-1}(V_1)$ does not contain any elements of $M_f$ whose rank is greater than or equal to the rank of $x$. (Otherwise by continuity there would exist a point $\hat{x} \in f^{-1}(y)$ such that $\rank(\hat{x}) > \rank(x)$.) Let $y_1 \in f(M_f) \cap (V_1 \setminus \left\{y\right\})$, and consider $f^{-1}(y_1)$. We again know that there exists an element $x_1 \in f^{-1}(y_1)$ such that for every other $x' \in f^{-1}(y_1)$, $\rank(x') \leq \rank(x_1)$. Since $f(X)$ is totally disconnected and zero-dimensional, we know there exists a clopen set $V_2 \subset V_1$ containing $y_1$ such that $f^{-1}(V_2)$ does not contain any elements of $M_f$ whose rank is greater than or equal to the rank of $x_1$. Note that from this process we are creating a decreasing sequence of ordinals (namely, the rank of each $x_n$). Since this sequence must be finite, we know for some step in this process, we will get to the point where $y_{j} \in M_f \cap (V_j \setminus \left\{y,y_1,...,y_{j-1}\right\})$ is such that the greatest rank of the elements in $f^{-1}(y_j)$ is $1$. That is, the points of $f^{-1}(y_j)$ must be isolated in $M_f$. By continuity of $f$, and the fact that $X$ is both limit point and sequentially compact, $y_j$ must be isolated in $f(M_f)$, a contradiction. Therefore $f(M_f)$ must be scattered. \end{proof}

\begin{lemma} Let $X$ be a totally disconnected, compact $T_2$ space cleavable over a totally disconnected separable LOTS $Y$. If there exists a continuous $f: X \rightarrow Y$ such that $M_f$ is scattered, then the rank of $f(M_f)$ must be less than $\omega_1$. \end{lemma}

\begin{proof} Assume for a contradiction that $\rank(f(M_f)) = \omega_1$. (If $\rank(f(M_f))> \omega_1$, just take an appropriate subspace.) Let $A$ be the set of elements $y \in f(X)$ such that $(f(M_f) \cup \left\{y\right\})^{\omega_1} = \left\{y\right\}$. We know $A$ is non-empty since $f(X)$ is compact; furthermore, $A$ is sequentially closed, and since $f(X)$ is first-countable, $A$ is a closed subset of $f(X)$. Note that $A \subset f(X) \setminus f(M_f)$.

Take $f(X) \setminus A$. This is open, and therefore made up of the union of open intervals. Since $f(X)$ is a compact LOTS, we know these open intervals are maximal. Consider one $(a,b)$ such that $(a,b) \cap f(M_f) \neq \emptyset$, and such that $a$ and $b$ are both in $A$. Since this is a non-empty open interval, and $f(X)$ is compact and totally disconnected, there exist gap points $c_L$ and $c_R$ such that $(a,b) = (a,c_L] \cup [c_R,b)$. Assume that $([c_,b) \cup \left\{b\right\})^{\omega} = \left\{b\right\}$. As $[c_R,b]$ is a totally disconnected compact separable LOTS, we may partition $[c_R,b)$ into countably many disjoint clopen intervals $V_m$, $m \in \omega$, such that $\bigcup_{m \in \omega} V_m = [c_R,b)$. 

Now since $b \in A$, we know that $[c_R,b) \cap f(M_f)$ is uncountable. Thus for at least one $m \in \omega$, $V_m \cap f(M_f)$ is uncountable as well. But by construction, $\rank(V_m \cap f(M_f)) < \omega_1$. Therefore, for some $\beta < \rank(V_m \cap f(M_f))$, $(V_m \cap f(M_f))^{\beta} \setminus (V_m \cap f(M_f))^{\beta+1}$ is uncountable as well. Around each point $x \in (V_m \cap f(M_f))^{\beta} \setminus (V_m \cap f(M_f))^{\beta+1}$ we may take an open interval around $x$, containing no other point of $(V_m \cap f(M_f))^{\beta} \setminus (V_m \cap f(M_f))^{\beta+1}$, thereby creating uncountably many disjoint open intervals, which contradicts the fact that $f(X)$ is separable.

Thus the rank of $f(M_f)$ must be less than $\omega_1$. \end{proof}

We are now ready to prove the main theorem of this section. Note that while we are assuming $Y$ is totally disconnected for now, we will be able to drop this assumption and prove the theorem holds for any separable $Y$.

\begin{thm} \label{scattered} Let $X$ be a totally disconnected, compact $T_2$ space cleavable over a totally disconnected separable LOTS $Y$. If there exists a continuous $f: X \rightarrow Y$ such that $M_f$ is scattered, then $X$ is a LOTS. \end{thm}

\begin{proof} We will prove this by transfinite induction on the Cantor-Bendixson rank of $f(M_f)$.\\

\textbf{Base Case}: If the rank of $f(M_f)$ is $0$, then it is empty, and the theorem is true vacuously. It would be useful, however, to exhibit the proof for the case where $\rank(f(M_f)) = 1$. Therefore let $f(M_f)^0$ be the last non-empty derived set of $f(M_f)$. Enumerate the elements of $f(M_f)$ as $y_j$, where $j \in v \subseteq \omega$. Let $U_j$ be a clopen set containing $y_j$ and no other $y_k$ for $k \neq j$. By Lemmas~\ref{partition} and \ref{less} we know we may find a LOTS $\hat{U}_j$ and a continuous function $f_j: f^{-1}(U_j) \rightarrow \hat{U}_j$ such that $M_{f_j} = \emptyset$. Let $\hat{Y} = Y$ with each $U_j$ replaced with $\hat{U}_j$, let $\hat{f}_j: f^{-1}(U_j) \rightarrow \hat{U}_j \subseteq \hat{Y}$ be identical to the function $f_j$, let $\hat{f}: X \setminus \bigcup_{j \in v} f^{-1}(U_j) \rightarrow \hat{Y}$ be identical to $f$, (note that $X \setminus \bigcup_{j \in v} f^{-1}(U_j)$ may be empty), and let $g: X \rightarrow \hat{Y}$ be defined as the following: $$g(x) = \begin{cases}
\hat{f}(x) & x \in X \setminus \bigcup_{j \in v} f^{-1}(U_j)\\
\hat{f}_j(x) & x \in f^{-1}(U_j) \end{cases}$$

Then $g$ is continuous, $M_g$ is empty, and since $g$ is an injective continuous function from $X$ to $\hat{Y}$, $X$ is a LOTS.\\

\textbf{Successor Case}: There are two successor cases to consider: when $\rank(f(M_f)) = \alpha+1$, where $\alpha$ is a successor ordinal, and when $\rank(f(M_f)) = \lambda +1$, where $\lambda$ is a limit ordinal. First assume $\rank(f(M_f)) = \alpha+1$, where $\alpha$ is a successor ordinal, and assume it is true that if the rank of $f(M_f)$ is $\alpha$, then $X$ is a LOTS. Since $\alpha$ is a successor ordinal, let it be equal to $\beta +1$. Enumerate the elements of $f(M_f)^{\alpha}$ as $y_j$, where $j \in v \subseteq \omega$. Since $f(X)$ is a totally disconnected LOTS, we may find clopen intervals $U_j$ containing $y_j$ and no other $y_k$ for $k \neq j$ such that $f(M_f) \subseteq \bigcup_{j \in v} U_j$. Now consider a single $U_j$. This clopen set has only one element of $f(M_f)^{\alpha} = f(M_f)^{\beta+1}$; label it $x_{\alpha,j}$. Enumerate the points of $(f(M_f)^{\beta} \cap U_j)  \setminus \left\{x_{\alpha,j}\right\}$ as $y_{j,m}$, where $m \in \omega$. Since $U_j$ is also a totally disconnected compact LOTS, we may find clopen intervals $V_{j,m} \subset U_j$ containing $y_{j,m}$ and no other $y_{j,k}$ for $k \neq m$. The rank of each $V_{j,m} \cap f(M_f)$ is $\beta+1 = \alpha$, therefore by the inductive hypothesis, we may find a clopen LOTS $\hat{V}_{j,m}$ and find a continuous function $f_{j,m}: f^{-1}(V_{j,m}) \rightarrow \hat{V}_{j,m}$ such that $M_{f_{j,m}} = \emptyset$. Let $Y_1$ be $Y$ with each $V_{j,m}$ replaced with $\hat{V}_{j,m}$. We may then find a continuous function $g_1$ from $X$ to $Y_1$ such that $M_{g_1} = f^{-1}(f(M_f)^{\alpha})$. We are now left with a situation where $\rank(g_1(M_{g_1})) = 1$. Therefore we know we may find a LOTS $Y_2$ and a continuous function $g_2: X \rightarrow Y_2$ such that $M_{g_2}= \emptyset$, making $X$ a LOTS.

Now let $\rank(f(M_f)) = \lambda +1$, where $\lambda$ is a limit ordinal, and assume we have shown that if the rank of $f(M_f)$ is $\alpha$, for $\alpha < \lambda+1$, then $X$ is a LOTS. Enumerate the elements of $f(M_f)^{\lambda}$ as $y_j$, where $j \in v \subseteq \omega$. Since $f(X)$ is a totally disconnected LOTS, we may find clopen intervals $U_j$ containing $y_j$ and no other $y_k$ for $k \neq j$ such that $f(M_f) \subseteq \bigcup_{j \in v} U_j$. Now consider $U_j$. This clopen set has only one element of $f(M_f)^{\lambda}$, namely $y_j$. For some $\alpha+1 < \lambda$, we know there exists a sequence of elements $\left\langle y_{j,n} \right\rangle$ contained in $(f(M_f)^{\alpha+1} \setminus f(M_f)^{\alpha+2}) \cap U$ converging to $y_j$. Since $U_j$ is a totally disconnected LOTS, and a clopen interval, we know we may partition $U_j$ into countably many disjoint clopen intervals, $V_{j,n}$, such that $y_{j,n} \in V_{j,n}$ for every $n \in \omega$, and $\bigcup_{n \in \omega} V_{j,n} \cup \left\{y_j \right\} = U_j$. Each $f(M_f) \cap V_{j,n}$ must have Cantor-Bendixson rank less than $\lambda+1$, therefore we know by the inductive hypothesis that we may find a LOTS $\hat{V}_{j,n}$ and a continuous function $g_{j,n}: f^{-1}(V_{j,n}) \rightarrow \hat{V}_{j,n}$ such that $M_{g_{j,n}} = \emptyset$. Let $Y_1$ be $Y$ with each $V_{j,n}$ replaced by $\hat{V}_{j,n}$, for every $n \in \omega$, and for every $j \in v$; let $g: X \rightarrow Y$ be equal to $f$ on $X \setminus (\bigcup_{j \in v} U_j \setminus \left\{y_j: j \in v \right\})$ (if it is not empty), and equal to $g_{j,n}$ on each $V_{j,n}$. Then $M_g = \left\{y_j: j \in v \right\}$, and we are now left with a situation in which the rank of $f(M_g) = 1$. Therefore we know $X$ is a LOTS.\\

\textbf{Limit Case}: Let the Cantor-Bendixson rank of $f(M_f)$ be equal to $\gamma$, where $\gamma$ is a limit ordinal, and assume we have shown that if $\rank(f(M_f)) < \gamma$, then $X$ is a LOTS.  

Let $A$ be the set of $x \in f(X) \setminus f(M_f)$ such that there exists a sequence of elements in $f(M_f)$, $\left\langle x_n \right\rangle$, converging to $x$, where $\sup{\left\{\rank(x_n)\right\}} = \gamma$. We know that $A$ is non-empty as $f(X)$ is compact, first-countable, and therefore sequentially compact. We also know from the way we have defined $A$, that it is sequentially closed. Thus $A$ is closed, as in a first-countable space, every sequentially closed subset is closed.  Also notice that $A$ is nowhere dense in $f(X)$.

Take $f(X) \setminus A$. This is open, and therefore made up of the union of open intervals. Since $f(X)$ is a compact LOTS, we know these open intervals are maximal. The set $f(M_f)$ must be contained within $X \setminus A$, and since $f(X)$ is separable, we know there may only be countably many of these maximal open intervals. Enumerate them as $(a_n,b_n)$, where $n \in v \subseteq \omega$.

Take a single maximal open interval, $(a_m,b_m)$. Assume both $a_m$ and $b_m$ are elements of $A$, and that the sequences that qualify $a_m$ and $b_m$ to be elements of $A$ are both contained within this open interval. Since this is a non-empty open interval, and $f(X)$ is compact and totally disconnected, there exist gap points $c_L$ and $c_R$ such that $(a_m,b_m) = (a_m,c_L] \cup [c_R,b_m)$. 

Now take $[c_R,b_m)$, and let $\left\langle y_{m,n} \right\rangle$ be a sequence contained in $[c_R,b_m) \cap f(M_f)$ that converges only to $b_m$. As $[c_R,b_m]$ is a totally disconnected, compact LOTS, we may partition $[c_R,b_m)$ into clopen intervals $V_{m,n}$, $n \in \omega$, such that $V_{m,n}$ contains $y_{m,n}$, and such that $\bigcup_{n \in \omega} V_{m,n} = [c_R,b_m)$. By construction, each $V_{m,n} \cap f(M_f)$ has Cantor-Bendixson rank less than $\gamma$. Therefore by the inductive hypothesis, we know there exists a LOTS $\hat{V}_{m,n}$ and a continuous function $f_{m,n}: f^{-1}(V_{m,n}) \rightarrow \hat{V}_{m,n}$ such that $M_{f_{m,n}} = \emptyset$. Let $Y_1$ be $Y$ with each $V_{m,n}$ replaced by $\hat{V}_{m,n}$, for every $m \in v$, $n \in \omega$. Then there exists a continuous function $g_1: X \rightarrow Y_1$, composed piecewise of functions identical to $f_{m,n}$, for all $m \in v$, $n \in \omega$, such that $M_{g_1} = \emptyset$, proving $X$ is a LOTS. \end{proof}

We have now proved that $X$ is a LOTS when $Y$ is totally disconnected, but what if $Y$ is not totally disconnected? The following lemma allows us to find a totally disconnected, separable LOTS $\hat{Y}$ that we may use instead of $Y$ in order to complete the proof of the theorem. Note that this lemma also looks very similar to Lemma~\ref{less}. While the statements differ only slightly, they are actually exhibiting very different properties. What Lemma~\ref{less} showed is that we may enlarge $Y$ so that $f$ is injective on more points of $X$; Lemma~\ref{td} shows that we may find a totally disconnected separable $\hat{Y}$ and a continuous $\hat{f}: X \rightarrow \hat{Y}$ such that we may implement Theorem~\ref{scattered} to show $X$ is a LOTS. To begin, we must state a new definition, and three introductory theorems.

\begin{definition} Let $X$ be a topological space, $Y = [a,b]$ a LOTS, and let $f$ be a continuous function from $X$ to $Y$. We say a point $y \in Y$ can be \textbf{separated} if there exists $x_1, x_2 \in X$ such that $f(x_1) = f(x_2) = y$, a space $\hat{Y} = [a,y_L] \cup [y_R,b]$, and a continuous function $g: X \rightarrow \hat{Y}$ such that if $h$ embeds $[a,y) \cup (y,b] \subset Y$ into $\hat{Y}$ in the obvious way, then $g(x) = h((f(x))$ when $f(x) \neq y$, $g(x) = y_L$ if $x \in \overline{f^{-1}([a,y))} \setminus f^{-1}([a,y))$, and $g(x) = y_R$ if $x \in \overline{f^{-1}((y,b])} \setminus f^{-1}((y,b])$. For example, if the Double Arrow Space were mapped onto the unit interval $[0,1]$ in the obvious way, then every point in $(0,1)$ can be separated. \end{definition}

The following two theorems may be found in \cite{Arhangelskii3} and \cite{Engelking} respectively.

\begin{thm} \label{1} Every first-countable compact scattered space is metrizable. \end{thm} 

\begin{thm} \label{2} Every uncountable compact metric space includes a closed dense-in-itself subspace. \end{thm} 

\begin{thm} \label{3} Every first-countable compact scattered space is countable. \end{thm}

\begin{proof} This follows from Theorems~\ref{1} and \ref{2}. \end{proof}

\begin{lemma} \label{td} Let $X$ be a totally disconnected compact $T_2$ space cleavable over a separable LOTS $Y$. If $f:X \rightarrow Y$ is a continuous function such that $M_f$ is countable, then there exists a totally disconnected, separable LOTS $\hat{Y}$ and a continuous $\hat{f}: X \rightarrow \hat{Y}$ such that $M_{\hat{f}} \subseteq M_f$. \end{lemma}

\begin{proof} Let $X = A \cup B$, where $A$ is perfect, $B$ is scattered, and $A \cap B = \emptyset$. If $A$ is empty, then by Theorems~\ref{3} and \ref{contributions}, we must have that $X$ is homeomorphic to a countable ordinal. Then our totally disconnected separable LOTS $\hat{Y}$ would be this countable ordinal, and our $\hat{f}$ would be a homeomorphism. 

If $A$ is non-empty, then let $D \subset f(X)$ be a connected component; we must have that $f(A) \cap D = D$. To see this, assume without loss of generality that $f(X)$ is connected. We will show that $f(A) = f(X)$. Assume for a contradiction that $f(A) \neq f(X)$. Since $A$ is perfect and $f$ is continuous, $f(A)$ must be closed, and $f(X) \setminus f(A)$ must be non-empty and open. Let $(a,b)$ be a non-empty open interval contained in $f(X) \setminus f(A)$, and let $C \subset (a,b)$ be a closed interval. We must have that $C$ is uncountable since it is perfect, and therefore $f^{-1}(C) \subset B$ is a compact, first-countable, scattered set which is uncountable. This contradicts Theorem~\ref{3}. 

Therefore, to prove this lemma, it is sufficient to assume $X$ is perfect.

There are three cases to consider: either $f(X)$ is totally disconnected already (and then we have completed the proof), $f(X)$ is connected, or $f(X)$ contains a connected component. If it is either of the latter two cases, we will prove this by showing that if $D$ is the set of points of $f(X)$ that can be separated, then $D$ is dense in the connected components of $f(X)$. 

Without loss of generality, we will assume $f(X) = [c,d]$ is connected. (If it were not connected, we could modify the proof to consider an arbitrary connected component of $f(X)$.) Now take $x_1, x_2 \in M_f$ such that $f(x_1) = f(x_2)$. Let $U_1$ and $U_2$ be clopen sets containing $x_1$ and $x_2$ respectively such that $U_2 = X \setminus U_1$. Then both $f(U_1)$ and $f(U_2)$ are closed in $f(X)$, and $f(X) \setminus f(U_1)$ is open. Therefore there exists a maximal open interval $(a_2,b_2)$ contained in $f(X) \setminus f(U_1)$.

Assume without loss of generality that $b_2 \in f(U_1) \cap f(U_2)$. (Note that $b_2$ may equal $f(x_1)$.) Since $M_f$ is countable, and this intersection is closed in $f(X)$, then by Theorem~\ref{contributions}, $f(U_1) \cap f(U_2)$ must be homeomorphic to a countable ordinal. For the sake of this example, we may assume $b_2$ is isolated in $f(U_1) \cap f(U_2)$, otherwise we may take another element that is isolated in this intersection. As $b_2$ is isolated in $f(U_1) \cap f(U_2)$, $f(X)$ is connected, and $X$ is dense-in-itself, this implies that there exists a maximal open subset of $f(X) \setminus f(U_2)$, namely $(b_2,c)$. That is, $(a_2,b_2) \subseteq f(X) \setminus f(U_1)$ is maximal, $b_2 \in f(U_1) \cap f(U_2)$, and $(b_2,c) \subseteq f(X) \setminus f(U_2)$. Therefore $b_2$ can be separated. 

First note that $f(M_f)$ must be dense in $f(X)$, as otherwise we would have a connected subset of $X$, a contradiction as $X$ is totally disconnected. Now to see that $D$ is dense in $f(X)$, let $(a,b)$ be an open interval in $f(X)$, and let $[a',b'] \subset (a,b)$ be closed. If we take $\hat{X} = f^{-1}([a',b'])$, the topology on $\hat{X}$ to be the subset topology, and $f|_{\hat{X}}: \hat{X} \rightarrow Y$ as our fixed continuous function, we may repeat the previous argument and find a $z \in f|_{\hat{X}}(M_{f|_{\hat{X}}}) \subset (a,b)$ that can be separated. This implies $D$ is dense in $f(X)$. 

Since $D \subseteq f(M_f)$, and $f(M_f)$ is countable, we may enumerate the points of $D$ as $y_n$, where $n \in v \subseteq \omega$. Let $Y_1$ be the LOTS created after we have separated $y_1$ into $y_{1,L}$ and $y_{1,R}$, with $f_1: X \rightarrow Y_1$; let $Y_{j+1}$ be the LOTS created after we have separated $y_{j+1}$, in which $y_{1,L} < y_{1,R}$, and $f_j$ is the continuous function mapping $X$ into $Y_j$. Let $\hat{Y}$ be the space created after all points $y_j$ have been separated into $y_{j,L}$ and $y_{j,R}$ for all $j \in v$, and let the order on $\hat{Y}$ be such that $y'_1 \leq y'_2$ if and only if for some $m \in v$, $y'_1 \leq y'_2$ in $Y_k$ for every $k \geq m$. Note that this relation preserves the order between those points of $f(X)$ that were not separated, and orders the points added during separation. This is obviously a linear order. Let the topology on $\hat{Y}$ be the linear order topology. Note that the space is totally disconnected since the points that can be separated are dense in $f(X)$.

We will now show that there exists a continuous function from $X$ to $\hat{Y}$. Informally, what this function will be doing is mapping fibers of points that have been separated to the appropriate places in $\hat{Y}$, and mapping the fibers of points that have no been separated to where they would ``normally'' go.

To now formally define $\hat{f}$, let $S$ be those elements $x \in X$ such that $f(x)$ maps onto a point that can be separated, and let the set of points of $\hat{Y}$ that were not created by separating points in $Y$ be written as $T$. Notice that $T$ can be embedded into $Y$. Let $h$ be such an embedding. Let $g: X \setminus S \rightarrow \hat{Y}$ be such that $g(x) = h^{-1}(f(x))$. Let $\hat{f}:X \rightarrow \hat{Y}$ be defined as $$\hat{f}(x) = \begin{cases}
g(x) & x \in X \setminus S\\
f_j(x) & x \in f^{-1}(y_j)\\ \end{cases}$$

Let us show this is continuous. 

Let $\hat{Y} = [d,e]$, and consider $[d,c) \subset \hat{Y}$. If $c$ was not separated in the construction of $\hat{Y}$, then $\hat{f}^{-1}([d,c)) = f^{-1}([d,c))$, which we know is open in $X$ by continuity of $f$. If, on the other hand, $c$ was separated, then let $c = c_L$. Then $\hat{f}^{-1}([d,c_L)) = f^{-1}([d,c))$, which is again open in $X$. If $c = c_R$, then $\hat{f}^{-1}([d,c)) = {f^{-1}}_j([d,c_R))$, which we know is open in $X$ by continuity of $f_j$. By a nearly identical argument, the sets $\hat{f}^{-1}((c,e])$ are open in $X$ as well. Therefore $\hat{f}$ is a continuous map from $X$ to a totally disconnected LOTS $Y$. It is obvious that $M_{\hat{f}} \subseteq M_f$. \end{proof}

\begin{thm} \label{result} Let $X$ be a totally disconnected, compact $T_2$ space cleavable over a separable LOTS $Y$. If there exists a continuous $f: X \rightarrow Y$ such that $M_f$ is scattered, then $X$ is a LOTS. \end{thm}

\begin{proof} This follows directly from Theorems~\ref{scattered} and \ref{td}. \end{proof}

\section{True for all $X$}

We have now shown that if $X$ is a totally disconnected compact $T_2$ space cleavable over a separable LOTS $Y$ such that for some continuous $f: X \rightarrow Y$ we have $M_f$ is scattered, then $X$ is a LOTS. In order to explain how we will use this result to prove that $X$ is a LOTS even when it is not totally disconnected, we must first state a theorem from \cite{buzyakova}:

\begin{thm} \label{buzyakova} If $C$ is a continuum cleavable over a LOTS $Y$, then $C$ is homeomorphic to a subspace of $Y$, and is also therefore a LOTS. \end{thm}

In this section, we will combine the results from Theorems~\ref{result} and \ref{buzyakova} to show $X$ is a LOTS. That is, let $X$ be a compact $T_2$ space cleavable over a separable LOTS $Y$. $X$ is formed by a combination of trivial connected components (single points), and non-trivial connected components. We will first show that if we remove the interior of each of the non-trivial connected components, the result is a closed, totally disconnected subspace of $X$, which by Theorem~\ref{result} is a LOTS under the subspace topology. We will then use this linear order, combined with the linear order on each connected component (given to us by Theorem~\ref{buzyakova}), to show that the topology derived from the combined linear order on $X$ is equivalent to the original topology on $X$.

The are two obstacles in our way, however. While the connected components of $X$ may be a LOTS under the subspace topology, we must first ensure that the interior of these connected components do not interact with the rest of the space (see Lemma~\ref{middle}), and second ensure that each family of connected components behaves as if $X$ were a LOTS (see Lemma~\ref{endpoint}). After we have proved these properties true, we show $X$ is a LOTS in Theorem~\ref{LOTS}, which is the main result of this paper.

We begin by citing another result from \cite{buzyakova}, and one result from \cite{Arhangelskii}.

\begin{thm} \label{connected} Let $X = [a, b]$ be a linearly ordered continuum. Let $f$ be a continuous mapping of $X$ onto a LOTS such that $f(a) = f(b)$. Let $c, d$ be elements in $X$ whose images are the two end-points of $f(X)$. Then for any $x \in X \setminus \left\{c,d\right\}$ there exists $y \in X \setminus \left\{x\right\}$ such that $f(x) = f(y)$.\end{thm}

\begin{lemma} \label{vanja} Let $A$ and $B$ be disjoint subsets of a set $X$, and let $\left\{f_{\alpha}:\alpha < \tau \right\}$ be a family of mappings of the set $X$ into sets $Y_{\alpha}$, where $\tau$ is an infinite cardinal number, and let us also assume that for every $\alpha < \tau$ the cardinality of the set $M_{\alpha} = \left\{x \in X \setminus (A \cup B) : f_{\alpha}(x) \in f_{\alpha}(X\setminus\left\{x\right\})\right\}$ is not less than $\tau$. Then there exist disjoint subsets $U$ and $V$ of $X$ such that $A \subset U$, $B \subset V$, and $f_{\alpha}(U) \cap f_{\alpha}(V) \neq \emptyset$ for every $\alpha < \tau$. \end{lemma}

\begin{lemma} \label{example} The space $X = ([0,1] \times \left\{\omega\right\}) \cup (\left\{{1\over2}\right\} \times \omega)$, with the subspace topology inherited from the product $[0,1] \times (\omega +1)$, is not cleavable over $\mathbb{R}$.\end{lemma}

\begin{proof} Arhangel'ski\u\i\ proved this to be true in \cite{Arhangelskii}, but it would useful for us to give an example as to why. Let $\mathbb{S}$ be the irrationals, and let $A = (\mathbb{S} \cap [0,{1\over4}]) \cup ({1\over4},{3 \over 4}) \cup (\mathbb{Q} \cap ({3 \over 4},1]$. Then no continuous $f: X \rightarrow \mathbb{R}$ can cleave apart $A$ from its complement. \end{proof} 

The following lemma is also well known.

\begin{lemma} \label{subset} If $X$ is a space cleavable over a LOTS $Y$, then every $D \subseteq X$ is also cleavable over $Y$. \end{lemma} 

\begin{lemma} \label{middle} Let $X$ be a compact $T_2$ space cleavable over a separable LOTS $Y$. Let $\left\{C_n = [a_n,b_n]: n \in \omega \right\}$ be a family of non-trivial connected components of $X$ such that the sequences $\left\langle a_n \right\rangle$ and $\left\langle b_n \right\rangle$, both converge to $x$ for some $x \in X$. If $\left\langle w_n \right\rangle$ is a sequence made up of elements such that $w_n \in (a_n,b_n)$ for every $n \in \omega$, then we must have that $\left\langle w_n \right\rangle$ converges to $x$ as well. \end{lemma}

\begin{proof} What this lemma is trying to show is that families of connected components of $X$ behave in the same way as if $X$ were a LOTS. 

Let $\left\{C_n = [a_n, b_n]: n \in \omega \right\}$ be a family of non-trivial connected components of $X$. Assume for a contradiction that both $\left\langle a_n \right\rangle$ and $\left\langle b_n \right\rangle$ converge to a single point $x$, but some sequence $\left\langle w_n \right\rangle$, where each $w_n$ belongs to a different connected component $C_n$, converges to a point $\hat{x} \neq x$. We will show that such an $X$ cannot be cleavable over a separable LOTS $Y$. 

All continuous functions from $X$ to $Y$ either cleave apart $x$ and $\hat{x}$, or they do not. Of those that do cleave apart $x$ and $\hat{x}$, if none of them were injective on each $C_n$, then by Theorem~\ref{connected} and Lemma~\ref{vanja} we may construct a set $A$ along which no continuous $f: X \rightarrow Y$ can cleave. Therefore there must exist a continuous $f$ that cleaves apart $x$ and $\hat{x}$, and such that every $f|_{C_m}$ is injective. 

Consider $f(x)$. Since $f$ cleaves apart $x$ and $\hat{x}$, we know $f(x) \neq f(\hat{x})$, and since $Y$ is Hausdorff, let $V_1$ and $V_2$ be disjoint open intervals containing $f(x)$ and $f(\hat{x})$ respectively. By continuity of $f$, and by assumption, $f(x)$ must also contain all but finitely many of the $f(a_n)$ and $f(b_n)$. Let $j$ be the least element of $\omega$ such that $f(a_m)$ and $f(b_m)$ are contained in $V_1$ for every $m > j$. As $f$ is injective on each $C_n$, this implies that if $f(a_m)$ and $f(b_m)$ are elements of $V_1$, then $f(C_m) \subset V_1$ as well. Therefore $f(C_m) \subset V_1$ for every $m > j$. But $V_2$ must contain all but finitely many of the $f(w_m)$. Thus for every open $V_1 \ni f(x)$ and open $V_2 \ni f(\hat{x})$, $V_1 \cap V_2 \neq \emptyset$. This contradicts the fact that $Y$ is Hausdorff. Therefore the sequence $\left \langle w_n \right\rangle$ must converge to $x$. \end{proof} 

\begin{lemma} \label{endpoint} Let $X$ be a compact $T_2$ space cleavable over a separable LOTS $Y$. If $C \subseteq X$ is a non-trivial connected component, and $x \in C$ is such that $x \in \overline{X \setminus C} \setminus (X \setminus C)$, then $x$ must be an endpoint of $C$.\end{lemma}

\begin{proof} Were $x$ not an endpoint of $C$, this situation would be indentical to the one described in Lemma~\ref{example}, and we know such a $C$ and $x$ would imply that $X$ is actually not cleavable over $Y$. Though we were assuming $Y$ to be $\mathbb{R}$ is Lemma~\ref{example} and we are not assuming that in this theorem, we may find a subset of $Y$ that cannot be cleaved apart from its complement for the same reason as to why the set $A$ in Lemma~\ref{example} could not be cleaved from its complement. Therefore $x$ must be an endpoint of $C$. \end{proof}

Up to this point, we have shown that $X$ behaves like a LOTS on certain subsets. All that is left to do is to use these lemmas to prove that the topology on $X$ is equivalent to a LOTS. The following theorem is the main result of this paper.

\begin{thm} \label{LOTS} Let $X$ be a compact $T_2$ space cleavable over a separable LOTS $Y$. Then $X$ is a LOTS. \end{thm}

\begin{proof}If $X$ is either totally disconnected, or connected, then we know by Theorems~\ref{buzyakova} and \ref{result} that $X$ is a LOTS. Therefore assume $X$ is neither totally disconnected, nor connected. Define $D \subseteq X$ to be those elements of $X$ that are in the interior of non-trivial connected components, and let $C = X \setminus D$. I first claim that $C$ is a closed, totally disconnected subset of $X$.

It is obvious by construction that $C$ is totally disconnected. To prove it is closed, it is sufficient to show that $C$ is sequentially closed, as $X$ is first-countable and therefore sequential. Let $\left\langle x_n \right\rangle$ be a sequence contained in $C$. If it does not converge, or contain a convergent subsequence, then it is obviously sequentially closed in $C$. Now assume it does converge, without loss of generality, to a single point $x$. If $x$ belongs to a non-trivial connected component, then by Lemma~\ref{endpoint} $x$ must be an endpoint of the connected component, and therefore a member of $C$. If $x$ does not belong to a non-trivial connected component, then it belongs to $C$ by assumption. Therefore $C$ is sequentially closed, and closed in $X$.

By Theorem~\ref{result}, $C$ is a LOTS under the subspace topology of $X$. We also know each connected component of $X$ has a topology equivalent to the linear order topology. Thus there is a linear order on $X$ that matches the linear orders on both $C$ and $D$. We claim the topology derived from this linear order on $X$, $\tau_O$, is equal to the original topology of $X$, $\tau_X$. 

To first show $\tau_X \subseteq \tau_O$, let $U$ be a basic open set in $\tau_X$. We must show for every $x \in U$, there exists a basic open interval in $\tau_O$ that is a subset of $U$. 

If $x \in U \cap D$, then by Theorem~\ref{buzyakova} we may find an open interval containing $x$ contained in $U$. Now let $x \in U \cap C$. Without loss of generality, let $x$ be the left endpoint of some non-trivial connected component. (If it is not part of a connected component, then we may modify the proof.) By Theorem~\ref{buzyakova}, we know we may find an element of the connected component, $z$, such that $(x,z) \subset U$. Now if $x$ is isolated from the left within the linear order on $X$, then there exists some greatest $x' < x$, and thus $x \in (x',z) \subset U$. If $x$ is not isolated from the left, then by Theorem~\ref{result}, we know there exists a $y \in C$ such that $(y,x) \cap C$ must be contained in $U$. We claim $(y,z)$ must be contained within $U$.

Assume for a contradiction that $(y,z)$ is not a subset $U$. Since we know $(y,z) \cap C$ is contained within $U$, the only elements preventing $(y,z)$ from being within $U$ are the elements of the non-trivial connected components. This implies that there is a countable family of connected components $\left\{C_n = [a_n,b_n]: n \in \omega\right\}$ contained within $(y,z)$ such that $\left\langle a_n \right\rangle$ converges to $x$, but such that there exists a sequence of elements $\left\langle w_n \right\rangle$, each $w_n$ contained within a different connected component, which converges to a point other than $x$. By Lemma~\ref{middle} this is impossible. Thus $x \in (y,z) \subset U$, and $\tau_X \subseteq \tau_O$.

To show $\tau_O \subseteq \tau_X$, it is sufficient to show that $[-\infty,b)$ and $(a,\infty]$ are both open in $\tau_X$. Let $b \in C$ such that $b$ is, without loss of generality, the left endpoint of a non-trivial connected component. (Again, if it is not an endpoint of some connected component, we may modify the proof accordingly.) We will show that $[b,\infty]$ is closed in $\tau_X$, and thus $X \setminus [b,\infty] = [-\infty,b)$ is open. To show $[b,\infty]$ is closed in $X$, it is enough to show it is sequentally closed.

Let $\left\langle x_n \right\rangle$ be a sequence contained in $[b,\infty]$ that converges to a single point $x$. ($X$ is first-countable and compact thus sequentially compact.) If $x \in D$, then all but finitely many points of the sequence must be contained within the same connected component to which $x$ belongs. Thus if $\left\{x_n: n \in \omega \right\} \subseteq [b,\infty]$, then $x$ must be contained within $[b, \infty]$ as well. If, however, $x \in C$, then assume without loss of generality it is the left endpoint of a non-trivial connected component. 

If $\left\langle x_n \right\rangle$ is a monotonically decreasing sequence, we would have the same case as when $x \in D$, so $x$ must be contained within $[b, \infty]$ as well. Therefore assume the sequence is monotonically increasing. If all but finitely many of these elements are contained within $C$, then by Theorem~\ref{result} we know if $\left\{x_n: n \in \omega \right\} \subseteq [b,\infty]$, then $x$ must be contained within $[b, \infty]$. Assume, however, that infinitely many of these elements are member of $D$; that is, they belong to non-trivial connected components. We know only finitely many of this sequence may belong to the same connected component (otherwise there exists a subsequence that converges to a point other than $x$), thus without loss of generality, assume each member of $D$ in this sequence belongs to a different connected component. From the way we have defined the order on $X$, we know the right end-point $y_n$ of the connected component to which $x_n$ belongs will be contained within $[b,\infty]$. We know this new sequence, $\left\langle y_n \right\rangle$ converges to $x$ as well. We are now left with the same situation we have just considered, in which there exists a monotonically increasing sequences of elements of $C$ converging to $x$; therefore we know if $\left\{x_n: n \in \omega \right\} \subseteq [b,\infty]$, then $x$ must be contained within $[b, \infty]$. Since we have considered all types of sequences contained in $[b,\infty]$, and have proven this set is sequentially closed, it must therefore be closed in $X$. This implies $[-\infty,b)$ is open in $X$. By a similar argument, we may show $(a,\infty]$ is open in $X$ as well. Thus $\tau_O \subseteq \tau_X$, implying $\tau_O = \tau_X$ and that $X$ is a LOTS.\end{proof}

\section{Conclusion}

In \cite{Arhangelskii}, Arhangel'ski\u\i\ stated one motivation for cleavability as a generalization of injective mappings. That is, if an injective map from a compact space $X$ to a Hausdorff space $Y$ gives us as much information about $X$ as we have about $Y$, then how much can we find out about a compact $X$ if it is cleavable over a Hausdorff $Y$. As a motivating question, he asks if an infinite compactum cleavable over a LOTS is a LOTS itself. 

We have now provided a partial answer to Arhangelski\u\i\'s question, showing that if $X$ is a compact space cleavable over a separable LOTS $Y$ such that there exists a continuous $f: X \rightarrow Y$ with $M_f$ scattered, then $X$ is a LOTS. This gives us a great amount of information about $X$, as many papers have been written detailing properties of compact separable LOTS, some of which can be found in \cite{lb} and \cite{kunen}. The following questions, however, still remain open:

\begin{question} Let $X$ be a compact $T_2$ space cleavable over a separable LOTS $Y$. Is $X$ homeomorphic to a subspace of $Y$? \end{question}

\begin{question} Let $X$ be a compact $T_2$ space cleavable over a separable LOTS $Y$. Must there exist a continuous $f: X \rightarrow Y$ such that $M_f$ is countable and scattered in $X$? \end{question}

\nocite{*}
\bibliographystyle{nar}
\bibliography{scattered}

\begin{thebibliography}{1}

\bibitem{splitting}
Arkhangel'ski\u\i\, A.~V. and Shakhmatov, D.~B. (1990)
{\em Journal of Mathematical Sciences} {\bf 50(2)}, 1497 -- 1512.

\bibitem{Arhangelskii}
Arhangel'ski\u\i\, A.~V. (1992)
{\em Annals of the New York Academy of Sciences} {\bf 659}, 18--28.

\bibitem{Arhangelskii2}
Arhangel'ski\u\i\, A.~V. (1993)
{\em Topology and its Applications} {\bf 54}, 141--163.

\bibitem{buzyakova}
Buzyakova, R.~Z. (2004)
{\em Proceedings of the American Mathematical Society} {\bf 132(7)},
  2171--2181.

\bibitem{contributions}
Mazurkiewicz, S. and Sierpi\'nski, W. (1920)
{\em Fundamenta Mathematicae} {\bf 1(1)}, 17--27.

\bibitem{Arhangelskii3}
Arhangel'ski\u\i\, A.~V. (1963)
{\em Uspehi Mat. Nauk} {\bf 18(5)}, 139--145.

\bibitem{Engelking}
Engelking, R. (1989)
General Topology,
Helderman, Berlin.

\bibitem{lb}
Lutzer, D.~J. and Bennett, H.~R. (1969)
{\em Proceedings of the American Mathematical Society} {\bf 23(3)}, 664--667.

\bibitem{kunen}
Kunen, K. (2009)
{\em Topology and its Applications} {\bf 156}, 1199--1215.

\end{thebibliography}

\end{document}